\documentclass[10pt, reqno]{amsart}

\usepackage{graphicx}
\usepackage{hyperref}
\usepackage{amssymb}
\usepackage{cite}
\usepackage{amsmath}
\usepackage{latexsym}
\usepackage{amscd}
\usepackage{amsthm}
\usepackage{mathrsfs}
\usepackage{url}

\newtheorem{theorem}{Theorem}
\newtheorem{lemma}{Lemma}
\newtheorem{proposition}{Proposition}
\newtheorem{definition}{Definition}

\newtheorem{remark}{Remark}

\numberwithin{equation}{section}

\begin{document}
\title[Rigidity of umbilical hypersurfaces]{Weighted Hsiung-Minkowski formulas and \\ rigidity of umbilical hypersurfaces}

\author{Kwok-Kun Kwong}
\address{Department of Mathematics, National Cheng Kung University, Tainan City 701, Taiwan}
\email{kwong@mail.ncku.edu.tw}

\author{Hojoo Lee}
\address{Center for Mathematical Challenges, Korea Institute for Advanced Study, Seoul 02455, Korea}
\email{momentmaplee@gmail.com}

\author{Juncheol Pyo}
\address{Department of Mathematics, Pusan National University, Busan 46241, Korea}
\email{jcpyo@pusan.ac.kr}

\begin{abstract}
We use the weighted Hsiung-Minkowski integral formulas and Brendle's inequality to show new rigidity results. First, we prove Alexandrov type results for closed embedded hypersurfaces with radially symmetric higher order mean curvature in a large class of Riemannian warped product manifolds, including the Schwarzschild and Reissner-Nordstr\"{o}m spaces, where the Alexandrov reflection principle is not available. Second, we prove that, in Euclidean space, the only closed immersed self-expanding solitons to the weighted generalized inverse curvature flow of codimension one are round hyperspheres.
\end{abstract}


\maketitle

\section{Motivation and Main Results}

A. D. Alexandrov \cite{Ale1956, Ale1962} proved that the only closed hypersurfaces of constant
(higher order) mean curvature embedded in ${\mathbb{R}}^{n \geq 3}$ are round hyperspheres.
The embeddedness assumption is essential. For instance, ${\mathbb{R}}^{3}$ admits immersed tori with constant mean curvature, constructed by U. Abresch \cite{Abr1987} and
H. Wente \cite{Wen1986}. R. C. Reilly \cite{Rei1977} and A. Ros \cite{Ros1987, Ros1988} presented alternative proofs, employing
the Hsiung-Minkowski formula. See also Osserman's wonderful survey \cite {Oss1990}.

In 1999, S. Montiel \cite{Mon1999} established various general rigidity results in a class of warped product manifolds, including
the Schwarzschild manifolds and Gaussian spaces. Some of his results require the additional assumption that the closed hypersurfaces are star-shaped
with respect to the conformal vector field induced from the ambient warped product structure.
As a corollary \cite[Example 5]{Mon1999}, he also recovers Huisken's theorem \cite{Hui1990} that the closed, star-shaped,
self-shrinking hypersurfaces to the mean curvature flow in ${\mathbb{R}}^{n \geq 3}$ are round hyperspheres. In 2016, S. Brendle \cite{B2016} solved the
open problem that, in ${\mathbb{R}}^{3}$, closed, embedded, self-shrinking topological spheres to the mean curvature flow should be round. The embeddedness assumption is essential. Indeed, in 2015, G. Drugan \cite{Dru2015} employed the shooting method to prove the existence of a self-shrinking sphere with self-intersections in ${\mathbb{R}}^{3}$.

In 2001, H. Bray and F. Morgan \cite{BM2002} proved a general isoperimetric comparison theorem in a class of warped product spaces, including Schwarzschild manifolds.
In 2013, S. Brendle \cite{B2013} showed that Alexandrov Theorem holds in a class of sub-static warped product spaces, including Schwarzschild and Reissner-Nordstr\"{o}m manifolds. S. Brendle and M. Eichmair \cite{BE2013} extended Brendle's result to the closed, convex, star-shaped hypersurfaces with constant higher order mean curvature. See also
\cite{Gim2015} by V. Gimeno, \cite{LX2016} by J. Li and C. Xia, and \cite{WW2016} by X. Wang and Y.-K. Wang.

In this paper, we provide new rigidity results (Theorem \ref{first main theorem}, \ref{second main theorem} and \ref{third main theorem}). First, we associate the manifold $M^{n \geq 3} = \left( {N}^{n-1} \times [0,\bar{r}), \bar{g} = dr^2 + h(r)^2 \, g_{N} \right)$, where $(N^{n-1}, g_N)$ is a compact manifold with constant curvature $K$.
As in \cite{B2013, BE2013}, we consider four conditions on the warping function $h: [0,\bar{r}) \to [0, \infty)$:
\begin{itemize}
\item[\textbf{(H1)}]\label{h1} $h'(0) = 0$ and $h''(0) > 0$.
\item[\textbf{(H2)}]\label{h2} $h'(r) > 0$ for all $r \in (0,\bar{r})$.
\item[\textbf{(H3)}]\label{h3} $2 \, \frac{h''(r)}{h(r)} - (n-2) \, \frac{K - h'(r)^2}{h(r)^2}$ is monotone increasing for $r \in (0,\bar{r})$.
\item[\textbf{(H4)}]\label{h4} For all $r \in (0,\bar{r})$, we have
$\frac{h''(r)}{h(r)} + \frac{K-h'(r)^2}{h(r)^2} > 0$.
\end{itemize}

Examples of ambient spaces satisfying all the conditions include the classical Schwarzschild and Reissner-Nordstr\"{o}m manifolds \cite[Section 5]{B2013}.

\begin{theorem}\label{first main theorem}
Let $\Sigma$ be a closed $k$-convex ($H_k>0$) hypersurface embedded in ${M}^{n \geq 3}$.
Let $\{b_j(r)\}_{j=1}^k$ and $\{c_j(r)\}_{j=1}^k$ be two families of monotone increasing, smooth, non-negative functions. Suppose
$$
\sum_{j=1}^{k} \left(b_j(r) H_j +c_j(r)H_1 H_{j-1}\right)=\eta(r)
$$
for some smooth positive radially symmetric function $\eta(r)$
which is monotone decreasing in $r$.
\begin{enumerate}
\item \label{item 1}
$k=1:$ Assume (H1), (H2), (H3). Then $\Sigma$ is umbilical.
\item $k \in \{2, \cdots, n-1\}:$ Assume (H1), (H2), (H3), (H4). If $\Sigma$ is star-shaped, then it is a slice ${N}^{n-1} \times \left\{r_{0} \right\}$ for some constant $r_{0}$.
\end{enumerate}
\end{theorem}

Theorem \ref{first main theorem} contains two special cases worth mentioning: (i) $H_k=\eta(r)$ and (ii) $H_1 H_{k-1}=\eta(r)$, where $\eta(r)$ is a monotone decreasing function. The second case can be regarded as a ``non-linear'' version of the Alexandrov theorem and seems to be a new phenomenon. The same result also applies to the space forms $\mathbb R^n$, $\mathbb H^n$ and $\mathbb S^n_+$ (open hemisphere) without the star-shapedness assumption (Theorem \ref{thm: alex2}).

In general, the monotonicity assumptions on coefficient functions $b_j(r), c_j(r)$ and $\eta(r)$ cannot be dropped. Indeed, as in Remark \ref{rmk: counterexample},
we can show the existence of a thin torus in ${\mathbb{R}}^{3}$, such that its mean curvature function only depends on the radial distance $r$ from the origin and is monotone increasing in $r$.

We also prove the following general rigidity result for linear combinations of higher order mean curvatures, with less stringent assumptions on the ambient space
and a stronger assumption that the immersed hypersurfaces are star-shaped.

\begin{theorem}\label{second main theorem}
Suppose $(M^{n \geq 3}, \bar g)$ satisfies (H2) and (H4). Let $\Sigma$ be a closed star-shaped $k$-convex ($H_k>0$) hypersurface immersed in $M^n$, $\{a_i(r)\}_{i=1}^{l-1}$ and $\{b_j(r)\}_{j=l}^k$ ($2\le l<k\le n-1$) be a family of monotone decreasing, smooth, non-negative functions and a family of monotone increasing, smooth, non-negative functions respectively (where at least one $ a_i(r) $ and one $ b_j(r) $ are positive). Suppose
$$\sum_{i=1}^{l-1}a_i(r)H_i= \sum_{j=l}^{k} b_j(r) H_j.$$
Then $\Sigma$ is totally umbilical.
\end{theorem}
Theorem \ref {second main theorem} contains the case where $\frac{H_k}{H_l}=\eta(r)$ for some monotone decreasing function $\eta$ and $k>l$. We notice that the same result also applies to the space forms $\mathbb R^n$, $\mathbb H^n$ and $\mathbb S^n_+$ (open hemisphere) without the star-shapedness assumption (Theorem \ref{thm: alex4}).
Our result extends \cite[Theorem B]{Koh2000} by S.-E. Koh, \cite[Corollary 3.11]{Kwo2016} by the first named author, and
\cite[Theorem 11]{WX2014} by J. Wu and C. Xia. The monotonicity assumptions on $a_i(r)$ and $b_j(r)$ cannot be dropped, see Remark \ref{rmk: counterexample}.

We next prove, in Section \ref{sec: last}, a rigidity theorem for self-expanding solitons to the \textit{weighted generalized inverse curvature flow} in Euclidean
space ${\mathbb{R}}^{n \geq 3}$:
\begin{equation} \label{gICFintro}
\frac{d}{dt} \mathcal{F}=\sum_{0\le i<j\le n-1}a_{i,j}
{\left( \frac{H_i}{H_j} \right)} ^{\frac{1}{j-i}} \nu,
\end{equation}
where the weight functions $\{a_{i,j}(x)\mid 0\le i<j\le n-1\}$ are non-negative functions on the hypersurface satisfying $\displaystyle \sum_{0\le i<j\le n-1}a_{i,j}(x)=1$. Here, $\nu$ denotes the outward pointing unit normal vector field and $H_j$ is the $j$-th normalized mean curvature.
For example, when $a_{i,j}=1$ for some pair $(i, j)$, we have the generalized inverse curvature flow:
\begin{equation} \label{gICFintro}
\frac{d}{dt} \mathcal{F}= {\left( \frac{H_i}{H_j} \right)} ^{\frac{1}{j-i}} \nu,
\end{equation}
which generalizes the so called inverse curvature flow:
\begin{equation*} \label{ICFintro}
\frac{d}{dt} \mathcal{F} =\frac{H_{j-1}}{H_j} \nu.
\end{equation*}

The inverse curvature flow has been used to prove various geometric inequalities and rigidities: Huisken-Ilmanen \cite{HI2001}, Ge-Wang-Wu \cite{GWW2014}, Li-Wei-Xiong \cite{LWX2014}, Kwong-Miao \cite{KM2014}, Brendle-Hung-Wang \cite{BHT2016}, Guo-Li-Wu \cite{GLW2016}, and Lambert-Scheuer \cite{LS2016}. In Euclidean space, the long time existence of smooth solutions to the generalized inverse curvature flow (\ref{gICFintro}) was proved by Gerhardt in \cite{G1990} and by Urbas in \cite{U1990}, under some natural conditions on the initial closed hypersurface. Furthermore, they showed that the rescaled hypersurfaces converge to a round hypersphere as $ t \rightarrow \infty$.

\begin{theorem} \label{third main theorem}
Let $\Sigma$ be a closed hypersurface immersed in ${\mathbb{R}}^{n \geq 3}$. If $\Sigma$ is a self-expander
to the weighted generalized inverse curvature flow, then it is a round hypersphere.
\end{theorem}

In the proof of our main results, we shall use several integral equalities and inequalities. Theorem \ref{first main theorem} requires the embeddedness assumption as in the classical Alexandrov Theorem and is proved for the space forms in \cite{Kwo2016}. Theorem \ref{second main theorem} and \ref{third main theorem} require no embeddedness assumption. Theorem \ref{third main theorem} is proved in \cite{DLW2015} for the inverse mean curvature flow.

\section{Preliminaries} \label{Preliminaries}

Let $(N^{n-1}, g_N)$ be an $(n-1)$-dimensional compact manifold with constant curvature $K$. Our
ambient space is the warped product manifold $M^{n \geq 3} = {N}^{n-1} \times [0,\bar{r})$ equipped
with the metric $\bar{g} = dr^2 + h(r)^2 \, g_{N} $. The precise conditions on the warping function $h$ will be stated separately for each result.

In this paper, all hypersurfaces we consider are assumed to be connected, closed, and orientable. On a given hypersurface $\Sigma$ in $M$, we define the normalized $k$-th mean curvature function
\begin{eqnarray}
H_k:=H_k(\Lambda)=\frac{1}{\binom{n-1}{k}}\sigma_k(\Lambda),
\end{eqnarray}
where $\Lambda=(\l_1,\cdots,\l_{n-1})$ are the principal curvature functions on $\Sigma$ and the homogenous polynomial $\sigma_k$ of degree $k$ is the $k$-th elementary symmetric function
\[
\sigma_k(\Lambda)=\sum_{i_1<\cdots<i_{k}}\lambda_{i_1}\cdots\lambda_{i_k}.
\]
We adopt the usual convention $\sigma_0=H_{0}=1$.

\begin{definition}[\textbf{Potential function and conformal vector field}]
In our ambient warped product manifold $M$, we define the potential function $f(r) = h'(r)>0.$
We define the vector field $X = h(r) \, \frac{\partial}{\partial r}=\overline \nabla \psi$,
where $\psi'(r)=h(r)$ and $\overline \nabla$ is the connection on $M$. We note that it is conformal:
$\mathcal L_X \overline g = 2 f\overline g$ \cite[Lemma 2.2]{B2013}.
\end{definition}

\begin{definition}[\textbf{Star-shapeness}]
For a hypersurface $\Sigma$ oriented by the outward pointing unit normal vector field $\nu$, we say that it is star-shaped when $\langle X, \nu \rangle \ge 0$.i
\end{definition}
A useful tool in studying higher order mean curvatures is the $k$-th
Newton transformation $T_k: T\Sigma \rightarrow T \Sigma$ (cf. \cite{reilly1973variational, Rei1977}).
If we write
\[
T_k ( e_j ) =\sum_{i=1}^{n-1} ( T_k )_j^i e_{i},
\]
then $ (T_k) _j^i $ are given by
$$
{(T_k)}_j^{\,i}= \frac 1 {k!}
\sum_{\substack{1 \le i_1,\cdots, i_k \le n-1\\ 1\le j_1, \cdots, j_k \le n-1}}
\delta^{i i_1 \ldots i_k}_{j j_1 \ldots j_k}
A_{i_1}^{j_1}\cdots A_{i_k}^{j_k}
$$
where $(A_i^j)$ is the second fundamental form of $\Sigma$.
If $ \{e_i \}_{i=1}^{n-1} $ consist of eigenvectors of $A$ with
\[
A (e_j) = \lambda_j e_{j},
\]
then we have
\[
T_k (e_j) = \Lambda_j e_{j},
\]
where
\begin{equation}\label{eq-Lambda-i}
\Lambda_j = \sum_{\substack{1 \le i_1 < \cdots < i_k \le n-1, \\
j \notin \{i_1, \cdots, i_k \}}} \lambda_{i_1} \cdots \lambda_{i_k} =\sigma_k(\lambda_1, \cdots, \lambda_{j-1}, \lambda_{j+1}, \cdots, \lambda_{n-1}).
\end{equation}
One also defines $T_0 = \mathrm{Id}$, the identity map.
We have the following basic facts:
\begin{lemma} \label{T positive}
Let $\Sigma$ be a closed hypersurface in a warped product manifold $M^n$ satisfying the condition (H2).
\begin{enumerate}
\item On $\Sigma$, there is an elliptic point, where all principal curvatures are positive.
\item \label{item: p convex}
Assume that $\Sigma$ is $p$-convex $(H_{p}>0)$. Then the following assertions hold
\begin{enumerate}
\item\label{item: a}
For all $k \in \{1, \cdots, p-1\}$, we have
$T_k>0$ and $H_k>0$. For any $j \in \{1,\cdots, n-1\}$, we have
$H_{k;j}:=H_k(\lambda_1, \cdots, \lambda_{j-1}, \lambda_{j+1}, \cdots, \lambda_{n-1})>0$.
\item\label{item: b}
If $1\le i<j\le p$, then $0<\frac{H_{i-1}}{H_i}\le \frac{H_{j-1}}{H_j}$. The equality $\frac{H_{i-1}}{H_i}= \frac{H_{j-1}}{H_j}$ holds if and only if $\lambda_1=\cdots =\lambda_{n-1}$.
\item\label {item: c}
For $1\le i <j\le p$ and for any $l=\{1, \cdots, n-1\}$,
\begin{align*}
j H_i H_{j-1;l}> i H_j H_{i-1;l}.
\end{align*}
\end{enumerate}
\end{enumerate}
\end{lemma}

\begin{proof}
The first assertion is proved in \cite[Lemma 4]{LWX}.
As in the proof of \cite[Proposition 3.2]{BC1997},
$T_k>0$ when $k \in \{1, \cdots, p-1\}$,
which implies $$H_k(\lambda_1, \cdots, \lambda_{j-1}, \lambda_{j+1}, \cdots, \lambda_{n-1})>0$$ by \eqref{eq-Lambda-i}.
Also, $H_k=\frac{1}{(n-1-k){{n-1}\choose k}}\mathrm{tr}_\Sigma (T_k)>0$. For $1\le i<j\le p$, the classical Newton-Maclaurin inequality
$H_{i-1}H_j\le H_{j-1}H_i$
then gives
$0<\frac{H_{i-1}}{H_i}\le \frac{H_{j-1}}{H_j}$, with $\frac{H_{i-1}}{H_i}= \frac{H_{j-1}}{H_j}$  if and only if all $\lambda_l$ are the same.

We now show \eqref{item: c}. Let $\lambda=\lambda_l$, $m=n-1$, and $\sigma_{i;l}={{m-1}\choose i} H_{i;l}$. Note that $\sigma_{i}= \lambda \sigma_{i-1;l}+\sigma_{i;l}$, which implies
$$H_{i}=\frac{i}{m}\lambda H_{i-1;l}+\frac{m-i}{m}H_{i;l}.$$
Using this identity, \eqref {item: a}, and the Newton-Maclaurin inequality, we have
\begin{align*}
&j H_i H_{j-1;l}- i H_j H_{i-1;l}\\
=& j \left(\frac{i}{m}\lambda H_{i-1;l}+\frac{m-i}{m}H_{i;l}\right)H_{j-1;l}
- i \left(\frac{j}{m}\lambda H_{j-1;l}+\frac{m-j}{m}H_{j;l}\right) H_{i-1;l}\\
=& \frac{j(m-i)}{m}H_{i;l}H_{j-1;l}- \frac{i(m-j)}{m}H_{j;l}H_{i-1;l}\\
=&
(j-i)H_{j-1;l}H_{i-1;l}+
\frac{i(m-j)}{m}(H_{i;l} H_{j-1;l}-H_{j;l}H_{i-1;l})\\
>&0.
\end{align*}
\end{proof}
For the reader's convenience, let us also record the following Heintze-Karcher-type inequality due to Brendle \cite[Theorem 3.5 and 3.11]{B2013}, which is crucial in our proof of Theorem \ref{first main theorem}.
\begin{proposition}[\textbf{Brendle's Inequality}] \label{prop brendle}
Suppose the warped product manifold $(M,\bar g)$ satisfies (H1), (H2), and (H3). Let $\Sigma$ be a closed hypersurface embedded in $(M,\bar g)$ with positive mean curvature. Then $$\int_\Sigma \frac{f}{H_1}\ge \int_\Sigma \langle X, \nu\rangle.$$
The equality holds if and only if $\Sigma$ is umbilical. If, futhermore, (H4) is satisfied, then $\Sigma$ is a slice $N\times \{r_0\}$.
\end{proposition}

\section{Proof of Theorem \ref{first main theorem} and \ref{second main theorem}}

The following formulas will play an essential role in our proof.
\begin{proposition} \label{weighted HS}
Let $\phi$ be a smooth function on a closed hypersurface $\Sigma$ in a Riemannian manifold $M^n$.
\begin{enumerate}
\item \textrm{(\textbf{Weighted Hsiung-Minkowski formulas})} \label{item: HM}
For $k \in \{1, \cdots, n-1\}$, we have
\begin{equation} \label{weighted in}
\begin{split}
&\int_\Sigma \phi \left( f H_{k-1} - H_{k} \langle X, \nu\rangle \right)
+\frac{1}{k{{n-1}\choose k}}\int_\Sigma \phi \left(\mathrm{div}_\Sigma T_{k-1}\right)(\xi)\\
=&-\frac{1}{k{{n-1}\choose k}}\int_\Sigma \langle T_{k-1}(\xi), \nabla _\Sigma \phi\rangle.
\end{split}
\end{equation}
Here, $\xi=X^T$ is the tangential projection of the conformal vector field $X$ onto $T\Sigma$. (Note that $\mathrm{div}(T_0)=0$.)
\item \label{item: div}
Suppose $(M^n, \bar g)$ is the warped product manifold in Section \ref{Preliminaries}.
Then, for $k \in \{2, \cdots, n-1\}$,
\begin{equation}\label{eq: div}
(\mathrm{div}_\Sigma T_{k-1}) (\xi)
=-{{n-3}\choose {k-2}}\sum_{j=1}^{n-1} H_{k-2;j} \xi^j \mathrm{Ric}(e_j, \nu).
\end{equation}
Here, $\{e_j\}_{j=1}^{n-1}$ and $\{\lambda_j\}_{j=1}^{n-1}$ are the principal directions and principal curvatures of $\Sigma$, respectively, and
$H_{k-2;j} =H_{k-2} (\lambda_1, \cdots, \lambda_{j-1}, \lambda_{j+1}, \cdots, \lambda_{n-1})$.\\
If $\Sigma$ is star-shaped and (H4) is satisfied, then, for each $j\in \{1, \cdots, n-1\}$,
\begin{equation}\label{eq: xiRic}
-\xi^j\mathrm{Ric}(e_j, \nu)\ge 0 .
\end{equation}
\end{enumerate}
\end{proposition}

\begin{proof}
Let $\xi=X^T=X-\langle X, \nu\rangle \nu$, and recall that $X$ is conformal: $\mathcal L_X \overline g = 2 f\overline g$.
By \cite[Proposition 3.1]{Kwo2016}, we have
\[
\mathrm{div}_\Sigma (\phi T_{k-1}(\xi))
= (n-k) f \sigma_{k-1}\phi- k \sigma_k \phi\langle X, \nu\rangle+ \phi (\mathrm{div}_\Sigma T_{k-1})(\xi)+\langle T_{k-1}(\xi), \nabla _\Sigma \phi\rangle.
\]
Integrating this equation, we get \eqref{weighted in}.

We now show \eqref{item: div}.
Take a local orthonormal frame $\nu$, $e_{1}$, $\cdots$, $e_{n-1}$, so that $e_{1}$, $\cdots$, $e_{n-1}$ are the principal directions of $\Sigma$.
By the proof of \cite[Proposition 8]{BE2013} (note that $T^{(k)}$ in \cite{BE2013} is the $(k-1)$-th Newton transformation), we have
\[
(\mathrm{div}_\Sigma T_{k-1}) \xi =-\frac{n-k}{n-2}\sum_{j=1}^{n-1} \sigma_{k-2;j} \xi^j \mathrm{Ric}(e_j, \nu),
\]
where $\sigma_{k-2;j} =\sigma_{k-2} (\lambda_1, \cdots, \lambda_{j-1}, \lambda_{j+1}, \cdots, \lambda_{n-1})$, which is equivalent to \eqref{eq: div}.

It remains to show \eqref{eq: xiRic}.
As in \cite[Equation (2)]{B2013}, we compute
\[
\mathrm{Ric}
=-\left(\frac{h''(r)}{h(r)} -(n-2)\frac{K-h'(r)^2}{h(r)^2}\right)\overline g
-(n-2) \left(\frac{h''(r)}{h(r)} +\frac{ K-h'(r)^2}{h(r)^2}\right) {dr}^{2}.
\]
By the assumption (H4) and star-shaped condition $ \langle \frac{\partial}{\partial r}, \nu\rangle>0$, we have
\begin{align*}
-\xi^j\mathrm{Ric}( e_j, \nu)
=(n-2)\left(\frac{h''(r)}{h(r)}+\frac{K-h'(r)^2}{h(r)^2}\right)
\frac{(\xi^j)^2}{h(r)} \langle \frac{\partial}{\partial r}, \nu \rangle \geq 0.
\end{align*}
\end{proof}

\begin{theorem}[$=$ \textbf{Theorem \ref{first main theorem}}]\label{thm: alex1}
Suppose $(M^{n\geq 3}, \bar g)$ is the warped product manifold in Section \ref{Preliminaries}.
Let $\Sigma$ be a closed $k$-convex ($H_k>0$) hypersurface embedded in $M^n$.
Let $\{b_j(r)\}_{j=1}^k$ and $\{c_j(r)\}_{j=1}^k$ be two families of monotone increasing, smooth, non-negative functions. Suppose
\begin{equation}\label{eq: ass'}
\sum_{j=1}^{k} \left(b_j(r) H_j +c_j(r)H_1 H_{j-1}\right)=\eta(r)
\end{equation}
for some smooth positive radially symmetric function $\eta(r)$
which is monotone decreasing in $r$.
\begin{enumerate}
\item $k=1:$ Assume (H1), (H2), (H3). Then $\Sigma$ is umbilical.
\item $k \in \{2, \cdots, n-1\}:$ Assume (H1), (H2), (H3), (H4). If $\Sigma$ is star-shaped, then it is a slice ${N}^{n-1} \times \left\{r_{0} \right\}$ for some constant $r_{0}$.
\end{enumerate}
\end{theorem}

\begin{proof}
By dividing \eqref{eq: ass'} by $\eta(r)$, it suffices to prove the result in the case where
\begin{equation}\label{eq: normalized assumption}
\sum_{j=1}^{k} \left(b_j(r) H_j +c_j(r) H_1 H_{j-1}\right)=1.
\end{equation}
Assume first $j\in \{2, \cdots, k\}$.
By Proposition \ref{weighted HS} \eqref{item: div} and Lemma \ref{T positive}, $(\mathrm{div}_\Sigma T_{j-1}) \xi\ge 0$.
Therefore by the monotonicity of $b_j$, for each $j$ we have
\begin{equation}\label{eq: bj}
\begin{split}
&\int_\Sigma b_j(r) \left( f H_{j-1} - H_j \langle X, \nu\rangle \right)\\
\le&-\frac{1}{j{{n-1}\choose j}}\int_\Sigma \langle T_{j-1}(\xi), \nabla _\Sigma b_j\rangle
=-\frac{1}{j{{n-1}\choose j}}\int_\Sigma h(r)b_j'(r)\langle T_{j-1}(\nabla _\Sigma r), \nabla _\Sigma r\rangle
\le 0.
\end{split}
\end{equation}
Note that this inequality also holds for $j=1$ by Proposition \ref{weighted HS} (without assuming (H4) and star-shapedness).

Similarly, by the Newton-Maclaurin inequality, for each $j$ we have
\begin{equation}\label{eq: cj}
\begin{split}
& \int_\Sigma c_j(r) \left( f H_{j-1} - H_1H_{j-1} \langle X, \nu\rangle \right)\\
\le&\int_\Sigma c_j(r) \left( f H_{j-1} - H_j \langle X, \nu\rangle \right)
\le-\frac{1}{j{{n-1}\choose j}}\int_\Sigma \langle T_{j-1}(\xi), \nabla _\Sigma c_j\rangle
\le 0.
\end{split}
\end{equation}
Adding \eqref{eq: bj} and \eqref{eq: cj} together and then summing over $j$, using \eqref{eq: normalized assumption}, we have
\begin{equation}\label{ineq: 1}
\int_\Sigma \left(f\sum_{j=1}^k \left(b_j(r) H_{j-1}+c_j(r) H_{j-1}\right)-\langle X, \nu\rangle\right)\le 0.
\end{equation}
Note that $H_i>0$ for $i\le k$ by Lemma \ref{T positive}.
Multiplying the Newton-Maclaurin inequality $H_1 H_{j-1}\ge H_j$ by $b_j(r)$ and summing over $j$ gives
$$ H_1\sum_{j=1}^k b_j(r) H_{j-1}\ge \sum_{j=1}^k b_j(r) H_j. $$
Combining this with \eqref{ineq: 1}, we obtain the inequality
\begin{align*}
0\ge&
\int_\Sigma \left(\frac{f}{H_1}\sum_{j=1}^k \left(b_j(r) H_{j}+c_j(r) H_1H_{j-1}\right)-\langle X, \nu\rangle\right)
=\int_\Sigma \, \left( \frac{f}{H_{1}} - \langle X, \nu\rangle \right).
\end{align*}
However, Brendle's inequality (Proposition \ref{prop brendle})
is the reverse inequality
$$
\int_\Sigma \, \left( \frac{f}{H_{1}} - \langle X, \nu\rangle \right) \geq 0.
$$
These two inequalities imply the equality in Brendle's inequality. We conclude that $\Sigma$ is umbilical and that, in the case when the condition (H4) holds, it is a slice.
\end{proof}

Due to the Brendle's inequality \cite[Theorem 3.5]{B2013}  and the analogous, but simpler, weighted Hsiung-Minkowski integral formulas in the space forms (cf. \cite{Kwo2016}), without assuming the star-shapedness condition, we can use the idea of Theorem \ref{thm: alex1} to prove
\begin{theorem}\label{thm: alex2}
Let $\Sigma$ be a closed $k$-convex hypersurface embedded in $M^{n \geq 3}=\mathbb R^n$, $\mathbb H^n$ or $\mathbb S^n_+$ (open hemisphere). Let $r$ be the distance in $M^n$ from a fixed point $p_0\in M$ (chosen to be the center if $M=\mathbb S^n_+$). Let $\{b_j(r)\}_{j=1}^k$ and $\{c_j(r)\}_{j=1}^k$ be two families of monotone increasing, smooth, non-negative functions. Suppose
\begin{align*}
\sum_{j=1}^{k} \left(b_j(r) H_j +c_j(r)H_1 H_{j-1}\right)=\eta(r)
\end{align*}
for some smooth positive radially symmetric function $\eta(r)$
which is monotone decreasing in $r$. Then $\Sigma$ is a geodesic hypersphere.
\end{theorem}

\begin{remark}
Recently, Brendle's inequality is extended in several ways, for instance, see \cite{LX2016, LX2016b, WW2016, WWZ2014}.
We observe that the proof of \eqref{item 1} in Theorem \ref{first main theorem} works on more general warped product manifold $M^n = {N}^{n-1} \times [0,\bar{r})$, which admits the property that Brendle's inequality holds. For instance, the fiber ${N}^{n-1}$ can be a compact Einstein manifold, as in
Brendle's paper \cite{B2013}.
\end{remark}

We now give another rigidity result which contains as a special case where the ratio of two distinct higher order mean curvatures is a radial function.
\begin{theorem}[$=$ \textbf{Theorem \ref{second main theorem}}]\label{thm: alex3}
Suppose $(M^{n \geq 3}, \bar g)$  is the warped product manifold in Section \ref{Preliminaries} satisfying (H2) and (H4). Let $\Sigma$ be a closed star-shaped $k$-convex ($H_k>0$) hypersurface immersed in $M^n$, $\{a_i(r)\}_{i=1}^{l-1}$ and $\{b_j(r)\}_{j=l}^k$ ($2\le l<k\le n-1$) be a family of monotone decreasing, smooth, non-negative functions and a family of monotone increasing, smooth, non-negative functions respectively (where at least one $ a_i(r) $ and one $ b_j(r) $ are positive). Suppose
$$\sum_{i=1}^{l-1}a_i(r)H_i= \sum_{j=l}^{k} b_j(r) H_j.$$
Then $\Sigma$ is totally umbilical.
\end{theorem}

\begin{proof}
Let $\xi=X^T$ and $A_p=-\frac{1}{(n-1)(n-2)} \xi^p \mathrm{Ric}(e_p, \nu)$. Since we are assuming (H2) and (H4), we
can apply Lemma \ref{T positive} \eqref{item: p convex} and Proposition \ref{weighted HS} to have, for each $i$ and $j$,
\begin{equation}\label{ineq: 0}
\begin{split}
&\int_\Sigma a_i(r) \left( f H_{i-1} - H_i \langle X, \nu\rangle \right)+(i-1)\int_\Sigma a_i(r)\sum_{p=1}^{n-1}A_p H_{i-2;p}\\
=&-\frac{1}{i{{n-1}\choose i}}\int_\Sigma \langle T_{i-1}(\xi), \nabla _\Sigma a_i\rangle
=-\frac{1}{i{{n-1}\choose i}}\int_\Sigma h(r)a_i'(r)\langle T_{i-1}(\nabla _\Sigma r), \nabla _\Sigma r\rangle
\ge 0
\end{split}
\end{equation}
and
\begin{equation}\label{ineq: 1a}
\begin{split}
&\int_\Sigma b_j(r) \left( f H_{j-1} - H_j \langle X, \nu\rangle \right)+(j-1)\int_\Sigma b_j(r)\sum_{p=1}^{n-1}A_p H_{j-2;p}\\
=&-\frac{1}{j{{n-1}\choose j}}\int_\Sigma \langle T_{j-1}(\xi), \nabla _\Sigma b_j\rangle
=-\frac{1}{j{{n-1}\choose j}}\int_\Sigma h(r)b_j'(r)\langle T_{j-1}(\nabla _\Sigma r), \nabla _\Sigma r\rangle
\le 0.
\end{split}
\end{equation}
Summing \eqref{ineq: 0} over $i$ and \eqref{ineq: 1a} over $j$, and then taking the difference gives
\begin{equation}\label{ineq: Hkl}
\begin{split}
0=&\int \left(\sum_{j=l}^k b_j(r) H_j-\sum_{i=1}^{l-1} a_i(r) H_i\right)\langle X, \nu\rangle\\
\ge& \int_\Sigma f \left(\sum_{j=l}^k b_j(r) H_{j-1}-\sum_{i=1}^{l-1} a_i(r) H_{i-1}\right)\\
&+\int_\Sigma \sum_{p=1}^{n-1}A_p\left(\sum_{j=l}^k(j-1)b_j(r) H_{j-2;p}-\sum_{i=1}^{l-1} (i-1)a_i(r) H_{i-2;p}\right).
\end{split}
\end{equation}
Note that $H_j>0$ for $j\le k$ by Lemma \ref{T positive}.
Let $1\le i<j\le k$, then multiplying the Newton's inequality $H_i H_{j-1}\ge H_{i-1}H_j$ by $a_i(r) b_j(r)$ and summing over $i, j$ gives
$$
\sum_{i=1}^{l-1} a_i(r) H_i \sum_{j=l}^k b_j(r) H_{j-1}\ge \sum_{i=1}^{l-1} a_i(r) H_{i-1}\sum_{j=l}^k b_j(r) H_j.
$$
Since $\sum_{i=1}^{l-1}a_i(r)H_i=\sum_{j=l}^{k}b_j(r)H_j>0$, we deduce
\begin{equation}\label{ineq: lin comb}
\sum_{j=l}^k b_j(r)H_{j-1}\ge \sum_{i=1}^{l-1}a_i(r)H_{i-1}.
\end{equation}
Similar to \eqref{ineq: lin comb}, we can obtain from Lemma \ref {T positive} \eqref{item: c} the inequality
\begin{equation}\label{ineq: bH-aH}
\sum_{j=l}^k(j-1)b_j(r) H_{j-2;p}-\sum_{i=1}^{l-1} (i-1)a_i(r) H_{i-2;p} > 0.
\end{equation}
On the other hand, $A_p\ge 0$ by Proposition \ref {weighted HS}. Combining this with \eqref{ineq: bH-aH} and \eqref{ineq: lin comb}, we conclude that all the integrands in \eqref{ineq: Hkl} are zero. This implies \eqref{ineq: lin comb} is an equality and hence $\Sigma$ is totally umbilical by the Newton-Maclaurin inequality.
\end{proof}

Again, following the idea of Theorem \ref{thm: alex3}, we can use the weighted Hsiung-Minkowski integral formulas in the space forms to prove
\begin{theorem}\label{thm: alex4}
Let $\Sigma$ be a closed $k$-convex hypersurface immersed in $M^{n \geq 3}=\mathbb R^n$, $\mathbb H^n$ or $\mathbb S^n_+$ (open hemisphere). Let $r$ be the distance in $M^n$ from a fixed point $p_0\in M$ (chosen to be the center if $M=\mathbb S^n_+$). Let $\{a_i(r)\}_{i=1}^{l-1}$ and $\{b_j(r)\}_{j=l}^k$ ($2\le l<k\le n-1$) be a family of monotone decreasing, smooth, non-negative functions and a family of monotone increasing, smooth, non-negative functions respectively (where at least one $ a_i(r) $ and one $ b_j(r) $ are positive). Suppose
$$\sum_{i=1}^{l-1}a_i(r)H_i= \sum_{j=l}^{k} b_j(r) H_j.$$
Then it is a geodesic hypersphere.
\end{theorem}

\begin{remark}\label{rmk: counterexample}
We illustrate that the monotonicity condition on $a_i(r)$ and $b_j(r)$ in Theorem \ref{thm: alex3} and \ref{thm: alex4} cannot be dropped. For simplicity, we begin with the standard circular torus embedded in ${\mathbb{R}}^{3}$ given by
the level set
\[
{\left( \sqrt{\, {x_{1}}^2 + {x_{2}}^2 \,} - R_{1} \right)}^{2} + {x_{3}}^2 = {{R}_{2}}^{2},
\]
where the inner radius $R_{1}$ and outer radius $R_{2}$ satisfy $R_{2} < \frac{R_{1}}{2}$.
The normalized mean curvature function $H_{1}$ depends only on
$r=\sqrt{{x_{1}}^{2} + {x_{2}}^{2} +{x_{3}}^{2}}:$
\[
H_{1} = \frac{{R_{1}}^{2} - {r}^{2}} {{R_{1}}^3 - {R_{2}}^2 R_1 - R_1 r^2},
\]
which is increasing for $\displaystyle r \in \left[ R_{1}-R_{2}, R_{1}+R_{2}\right]$. Likewise, in ${\mathbb{R}}^{4}$, we can construct explicit counterexamples, by considering the hypersurface $\Sigma$ which is homeomorphic to ${\mathbb{S}}^{1} \times {\mathbb{S}}^{2}:$
\[
{\left( \sqrt{\, {x_{1}}^2 + {x_{2}}^2 \,} - R_{1} \right)}^{2} + {x_{3}}^2 +{x_{4}}^2 = {{R}_{2}}^{2}.
\]
When $\Sigma$ is sufficiently thin, in the sense that the inner radius $R_{1}$ and outer radius $R_{2}$ satisfy $R_{2} < \frac{R_{1}}{3}$, we can check that the two positive functions $H_{1}$ and $\frac{H_{2}}{H_{1}}$ depend only on the radial distance $r=\sqrt{{x_{1}}^{2} +\cdots +{x_{4}}^{2}}$, and are increasing for $\displaystyle r \in [\min_{x\in \Sigma}r(x),\max_{x\in \Sigma}r(x)] = \left[ R_{1}-R_{2},R_{1}+ R_{2} \right]$.
\end{remark}

\section{Proof of Theorem \ref{third main theorem}}\label{sec: last}
We consider the \textit{weighted generalized inverse curvature flow} in Euclidean space ${\mathbb{R}}^{n \geq 3}$:
\begin{equation}
\label{gICF now}
\frac{d}{dt} \mathcal{F}
= \sum_{0\le i<j\le n-1}a_{i,j}
{\left( \frac{{H}_{i}}{{H}_{j}} \right)} ^{\frac{1}{j-i}} \nu,
\end{equation}
where the weight functions $\{a_{i,j}(x)\mid {0\le i<j\le n-1}\}$ are non-negative functions on the hypersurface satisfying $\displaystyle \sum_{0\le i<j\le n-1} a_{i,j}(x)=1$. Here, $\nu$ denotes the outward pointing unit normal vector field and ${H}_{j}$ the $j$-th normalized mean curvature function. Let $k=\max\{j\mid a_{i,j}>0\textrm{ for some }i<j\}$ so that $H_k>0$. If $a_{i,j}=1$ and $j-i=1$, the evolution (\ref{gICF now}) is called the inverse curvature flow.

\begin{definition}
We say that a hypersurface $\Sigma$ with ${H}_{k}>0$ is a self-expander to the generalized inverse curvature flow
(\ref{gICF now}) if there exists a constant $\mu>0$ satisfying
\begin{equation} \label{gICF soliton here}
\sum_{0\le i<j\le k}a_{i,j} {\left(\frac{H_i}{H_j} \right)} ^{\frac{1}{j-i}} = \mu \langle X, \nu \rangle.
\end{equation}
\end{definition}
\begin{theorem}[$=$ \textbf{Theorem \ref{third main theorem}}] \label{third main theorem here}
Let $\Sigma$ be a closed hypersurface immersed in ${\mathbb{R}}^{n \geq 3}$. If $\Sigma$ is a self-expander
to the weighted generalized inverse curvature flow, then it is a round hypersphere centered at the origin.
\end{theorem}

\begin{proof} Let $\mathbf{p}= \langle X, \nu \rangle$ denote the support function on $\Sigma$.
We shall repeatedly use the classical Hsiung--Minkowski integral formulas \cite{Hs1954, Hs1956}
\begin{equation}\label{eq: HM}
\int_\Sigma H_j=\int_\Sigma H_{j+1}\mathbf{p}
\end{equation}
for hypersurfaces in Euclidean space.
By the $k$-convexity assumption $H_{k}>0$, Lemma \ref{T positive} (\ref{item: p convex}) guarantees that $H_{j}>0$ for $j \in \{0, \cdots, k \}$ and so $\mathbf{p}>0$ by \eqref{gICF soliton here}.

Assume first that $k\ge2$.
By Lemma \ref{T positive} (\ref{item: b}), we have for $0\le i<j\le k$,
$$ \frac{H_i}{H_{j-1}} =\prod_{m=i}^{j-2} \frac{H_m}{H_{m+1}} \le \prod_{m=i}^{j-2} \frac{H_{j-1}}{H_j} = {\left(\frac{H_{j-1}}{H_j}\right)}^{j-i-1} $$
and
$$ \frac{H_i}{H_j} =\prod_{m=i}^{j-1} \frac{H_m}{H_{m+1}} \ge\prod_{m=i}^{j-1} \frac{1}{H_1} = \frac{1}{{H_1}^{j-i}}. $$
It follows that
\begin{equation}\label{ineq: newton}
\left({\frac{H_i}{H_j}}\right)^{\frac{1}{j-i}} \leq \frac{H_{j-1}}{H_j}\;\; {and} \;\; {\left( \frac{H_i}{H_j} \right)} ^{\frac{1}{j-i}} \ge \frac{1}{H_1}=\frac{H_0}{H_1}.
\end{equation}
Therefore, by Lemma \ref{T positive} (\ref{item: b}) again,
\begin{equation}\label{ineq: newton1}
\mu\mathbf{p}= \sum_{i<j}a_{i,j} {\left(\frac{H_i}{H_j} \right)} ^{\frac{1}{j-i}}
\le \sum_{i<j}a_{i,j} \frac{H_{j-1}}{H_j}
\le \sum_{i<j}a_{i,j} \frac{H_{k-1}}{H_k}
=\frac{H_{k-1}}{H_k}
\end{equation}
and
\begin{equation}\label{ineq: newton2}
\mu\mathbf{p}= \sum_{i<j}a_{i,j} {\left(\frac{H_i}{H_j} \right)} ^{\frac{1}{j-i}}
\ge \sum_{i<j}a_{i,j} \frac{H_0}{H_1}
=\frac{H_0}{H_1}.
\end{equation}
The inequality \eqref{ineq: newton1} implies
$$\mu\int_\Sigma H_k \mathbf{p}\le \int_\Sigma H_{k-1},$$
which in turn implies $\mu \le 1$ by the Hsiung--Minkowski formula \eqref{eq: HM}.

On the other hand, \eqref{ineq: newton2} implies
$$\mu\int_\Sigma H_1\mathbf{p}\ge \int_\Sigma H_0$$
and hence $\mu \ge 1$ again by the Hsiung--Minkowski formula \eqref{eq: HM}.

We conclude that $\mu=1$ and all the inequalities in \eqref{ineq: newton} are all equalities. Therefore $\Sigma$ is umbilical and so is a round hypersphere, which is easily seen to be centered at the origin.

When $k=1$, \eqref{ineq: newton2} becomes an equality and hence $\mu=1$ by \eqref{eq: HM}. By the Newton-Maclaurin inequality, we have
\begin{equation}\label{ineq: newton3}
H_2 \mathbf{p}=\frac{H_2}{H_1}\le \frac{H_1}{H_0}= H_1.
\end{equation}
Integrating this inequality and comparing to
the Hsiung-Minkowski formula
\[
\int_\Sigma H_1 =\int_\Sigma H_2 \mathbf{p},
\]
we again deduce that \eqref{ineq: newton3} is an equality and hence $\Sigma$ is round.
\end{proof}

\textbf{Acknowledgements.} Part of our work was done while the three authors were visiting the National Center for Theoretical Sciences in Taipei, Taiwan in November 2015. We would like to thank the NCTS for their support and warm hospitality during our visit. We also would like to thank Yong Wei for an improvement of Theorem \ref{third main theorem} in an earlier version of this paper. The research of K.-K. Kwong is partially supported by Ministry of Science and Technology in Taiwan under grant MOST103-2115-M-006-016-MY3. J. Pyo is partially supported by the National Research Foundation of Korea (NRF-2015R1C1A1A02036514).

\bibliographystyle{Plain}

\end{document}